\documentclass{amsart}
\usepackage{amsmath, amssymb,epic,graphicx,mathrsfs,enumerate}
\usepackage[all]{xy}
\usepackage{color}
\usepackage{comment}
\usepackage{hyperref}

\usepackage{amscd}
\usepackage{newlfont}
\usepackage{amsthm}
\usepackage{amssymb}
\usepackage{latexsym}
\usepackage{longtable}
\usepackage{epsfig}
\usepackage{amsmath}
\usepackage{hhline}

 \DeclareMathOperator{\perm}{Sym}
 \DeclareMathOperator{\soc}{soc}
\DeclareMathOperator{\aut}{Aut} \DeclareMathOperator{\out}{Out}

 \DeclareMathOperator{\frat}{Frat}

\DeclareMathOperator{\sym}{Sym}

 \DeclareMathOperator{\psl}{PSL}

\DeclareMathOperator{\End}{End}

\newtheorem{thm}{Theorem}
\newtheorem{cor}[thm]{Corollary}
 
\newtheorem{prop}[thm]{Proposition}

\numberwithin{equation}{section}

\renewcommand{\footnote}{\endnote}
\newcommand{\ignore}[1]{}\makeglossary

\begin{document}
	\bibliographystyle{amsplain}
	\title[Independent generating sets]{Finite groups in which there are only two possible cardinalities for an independent generating set}

\author[A. Lucchini]{Andrea Lucchini}
\address{Andrea Lucchini, Dipartimento di Matematica \lq\lq Tullio Levi-Civita\rq\rq,\newline
 University of Padova, Via Trieste 53, 35121 Padova, Italy} 
\email{lucchini@math.unipd.it}
         
\author[P. Spiga]{Pablo Spiga}
\address{Pablo Spiga, Dipartimento di Matematica Pura e Applicata,\newline
 University of Milano-Bicocca, Via Cozzi 55, 20126 Milano, Italy} 
\email{pablo.spiga@unimib.it}


\begin{abstract}
 A generating set $S$ for a group $G$ is  independent if the subgroup generated by $S\setminus \{s\}$ is properly contained in $G$,  for all $s \in S.$ In this paper, we study a problem proposed by Peter Glasby: we investigate finite groups, where there are only two possible cardinalities for the independet generating sets.
\end{abstract}
	\maketitle

\section{Introduction}

The \textbf{\textit{minimal number of generators}} of a finite group $G$ is denoted  by $d(G).$  A generating set $S$ for a group $G$ is  \textbf{\textit{independent}} (sometimes
called \textbf{\textit{irredundant}}) if $$\langle S\setminus \{s\}\rangle < G,$$  for all $s \in S.$ Let $m(G)$ denote the \textit{\textbf{maximal size of an independent
generating set}} for $G.$

  The finite groups with $m(G)=d(G)$ are classified by Apisa and Klopsch:
\begin{thm}[{{Apisa-Klopsch,~\cite[Theorem 1.6]{ak}}}]\label{apisa} If $d(G)=m(G)$, then $G$ is soluble. Moreover, either 
\begin{itemize}
\item $G/\frat(G)$ is an elementary abelian $p$-group for some prime $p$, or 
\item $G/\frat(G)=PQ,$ where $P$ is an elementary abelian $p$-group and $Q$ is a non-trivial cyclic $q$-group, for distinct primes $p\neq q,$ 
such that $Q$ acts by conjugation faithfully on $P$ 
and $P$ (viewed as a module for $Q$) is a direct sum of $m(G)-1$ isomorphic copies of one simple $Q$-module.
\end{itemize}
\end{thm}

In view of the result of Apisa and Klopsch, in a recent paper, Glasby proposed to classify the finite groups $G$ with $$m(G) - d(G)\leq 1,$$ see~\cite[Problem 2.3]{sg}. Clearly, Theorem~\ref{apisa} gives a classification when $m(G)-d(G)=0$ and hence, the relevant case is when $m(G)-d(G)=1$.

 A nice result in universal algebra, due to Tarski and known with the name of  \textit{\textbf{Tarski irredundant basis theorem}} (see for example \cite[Theorem 4.4]{bs}), implies that, for every positive integer $k$ with $d(G)\leq k\leq m(G),$ $G$ contains an independent generating set of cardinality $k.$ So the condition  $m(G) - d(G)= 1$ is equivalent to the fact that there are only two possible cardinalities for an independent generating set of $G.$

Let $G$ be a finite group. We recall that the \textbf{\textit{socle}} of $G$ is the subgroup generated by the minimal normal subgroups of $G$; moreover,  $G$ is said to be \textit{\textbf{monolithic primitive}} if $G$ has a unique minimal normal subgroup and the Frattini subgroup $\frat(G)$ of $G$ is the identity. In this paper we prove the following two main results. 

\begin{thm}\label{uno}
	Let $G$ be a finite group with $\frat(G)=1$ and $m(G)=d(G)+1.$ If $G$ is not soluble, then $d(G)=2$, $G$ is a monolithic  primitive group and $G/\soc(G)$ is cyclic of prime power order.
\end{thm}

It was proved by Whiston and Saxl~\cite{sawhi} that $m(\psl(2,p))=3,$ for any prime $p$ with $p$ not congruent to $\pm 1$ modulo 8 or 10. In particular, as $d(S)=2$ for every non-abelian simple group, we deduce  there are infinitely many non-abelian simple groups $G$ with $m(G)=d(G)+1$. We also give examples of non-simple groups $G$ having $m(G)=d(G)+1$ in Section~\ref{sec:mon}.

\begin{thm}\label{thrm:sol}
	Let $G$ be a finite group with $\frat(G)=1$ and  $m(G)=d(G)+1.$ If $G$ is  soluble,	then one of the following occurs:
	\begin{enumerate}
		\item\label{sol1} $G\cong V\rtimes P$,
		where $P$ is a finite non-cyclic $p$-group and $V$ is an irreducible  $P$-module, which is not a $p$-group; in this case $d(G)=d(P)$;
		\item\label{sol2} $G\cong V^t \rtimes H,$ where $V$ is a faithful irreducible $H$-module, $m(H)=2$ and either $t=1$ or $H$ is abelian; in this case $d(G)=t+1$;
		\item\label{sol3} there exist two normal subgroups $N_1$, $N_2$  such that $1 \lneq N_1 \leq N_2,$ $N_1$ is an abelian minimal normal subgroup of $G,$ $N_2/N_1\leq \frat(G/N_1)$ and
		$G/N_2\cong V^t\rtimes H$, where $V$ is an irreducible $H$-module and $H$ is a non-trivial cyclic group of prime power order; in this case $d(G)=t+1.$
	\end{enumerate}
\end{thm}
In Section~\ref{sec:mon}, we  give examples of finite soluble groups $G$ with $m(G)=d(G)+1$, for each of the three possibilities arising in Theorem~\ref{thrm:sol}.

\section{Preliminary results}

Let $L$ be   a monolithic primitive group
and let $A$ be its unique minimal normal subgroup. For each positive integer $k$,
let $L^k$ be the $k$-fold direct
product of $L$.
The \textbf{\textit{crown-based power}} of $L$ of size  $k$ is the subgroup $L_k$ of $L^k$ defined by
$$L_k:=\{(l_1, \ldots , l_k) \in L^k  \mid l_1 \equiv \cdots \equiv l_k \ {\mbox{mod}} A \}.$$  

 In \cite{dv-l}, it is proved that, for every finite group $G$, there exists a monolithic group $L$ and a homomorphic image $L_k$ of $G$ such that 
\begin{enumerate}
	\item $ d(L/\soc L) < d(G) $, and 
	\item $d(L_k) =d(G).$ 
\end{enumerate} 
An $L_{k}$ with this property is called a \textbf{\textit{generating crown-based power}} for $G$.

In \cite{dv-l}, it is explained how  $d(L_{k})$ can be explictly computed in terms of $k$ and the structure of $L$.  
A key ingredient (when one wants to determine $d(G)$ from the behavior of 
the crown-based power homomorphic images of $G$)  is to evaluate, for each monolithic group $L$,  the maximal $k$ such that    $L_{k}$ is a   homomorphic image of $G$. 
This  integer $k$ 
arises from an  equivalence relation among the chief factors of $G$. In what follows, we give some details.

Given groups $G$ and $A$, we say that $A$ is a $G$-\textit{\textbf{group}} if $G$ acts on $A$ via automorphisms. In addition, $A$ is \textit{\textbf{irreducible}} if $G$ does not stabilise any non-trivial proper subgroups of $A$.
Two groups $A$ and $B$ are $G$-\textit{\textbf{isomorphic}} if there exists a group isomorphism $\phi: A\to B$ such that $\phi(g(a))=g(\phi(a))$ for all $a\in A$ and $g\in G.$
Following  \cite{paz},  we say that  
two irreducible $G$-groups $A$ and $B$  are  \textbf{\emph{$G$-equivalent}}, denoted $A  \sim_G B$, if there is an
isomorphism $\Phi: A\rtimes G \rightarrow B\rtimes G$ such that the following diagram commutes.

\begin{equation*}
	\begin{CD}
		1@>>>A@>>>A\rtimes G@>>>G@>>>1\\
		@. @VV{\phi}V @VV{\Phi}V @|\\
		1@>>>B@>>>B\rtimes G@>>>G@>>>1
	\end{CD}
\end{equation*}

Observe that two $G$-isomorphic $G$-groups are $G$-equivalent, and the converse holds if $A$ and $B$ are abelian. 

Let $A=X/Y$ be a chief factor of $G$. A complement $U$ of $A$ in $G$ is a subgroup of $G$ such that $$UX=G \hbox{ and } U \cap X=Y.$$ We say that   $A=X/Y$ is a \textbf{\textit{Frattini}} chief factor, if  $X/Y$ is contained in the Frattini subgroup of $G/Y$; this is equivalent to say that $A$ is abelian and there is no complement to $A$ in $G$. 

The  number $\delta_G(A)$  of non-Frattini chief factors $G$-equivalent to $A$   in any chief series of $G$  does not depend on the series. 

Now, 
we denote by  $L_G(A)$  the  \textit{\emph{monolithic primitive group  associated to $A$}},  
that is, 
$$L_G(A):=
\begin{cases}
	A\rtimes (G/\mathbf{C}_G(A)) & \text{ if $A$ is abelian}, \\
	G/\mathbf{C}_G(A)& \text{ otherwise}.
\end{cases}
$$
  
If $A$ is a non-Frattini chief factor of $G$, then $L_G(A)$ is a homomorphic image of $G$. 
More precisely,  there exists 
a normal subgroup $N$ such that $G/N \cong L_G(A)$ and $\soc (G/N) \sim_G A$. We identify $\soc( L_G(A))$ with  $A$, as $G$-groups.

 Consider now  all the normal subgroups $N$ of $G$ with the property that  $G/N \cong L_G(A)$ and $\soc (G/N) \sim_G A$: 
the intersection $R_G(A)$ of all these subgroups has the property that  $G/R_G(A)$ is isomorphic to the crown-based  power $(L_G(A))_{\delta_G(A)}$.  
Moreover, the socle $I_G(A)/R_G(A)$ of $G/R_G(A)$ is called the $A$-\textit{\textbf{crown}} of $G$ and it is  a direct product of $\delta_G(A)$ minimal normal subgroups $G$-equivalent to $A$.    

Note that, if  $L$ is monolithic primitive and $L_k$ is a homomorphic image of $G$ for some $k\geq 1$,  then  $L \cong L_G(A)$ for some non-Frattini chief factor $A$ of $G$ and $k \leq \delta_G(A)$. 
Furthermore, if $(L_G(A))_k$ is a generating crown-based power,  then so is   $(L_G(A))_{\delta_G(A)}$; in this case, we  say 
that  $A$  is a \textbf{\emph{generating chief factor}} for $G$. 

For an irreducible $G$-module $M$,   set
$$\begin{aligned}r_G(M):&=\dim_{\End_G(M)}M,\\ s_G(M)&:=\dim_{\End_G(M)} H^1(G,M), \\  t_G(M)&:=\dim_{\End_G(M)} H^1(G/\mathbf{C}_G(M),M).
	\end{aligned}$$
Notice that $$s_G(M)=t_G(M)+\delta_G(M),$$ see e.g.~\cite[1.2]{andrea2}. Now, define

$$h_{G}(M):= 
\begin{cases}
	\delta_G(M)&\text{ if $M$ is a  trivial } G\textrm{-module}, \\
	  \left\lfloor\frac{s_G(M)-1}{r_G(M)}\right\rfloor+2=\left\lfloor\frac{\delta_G(M)+t_G(M)-1}{r_G(M)}\right\rfloor+2&
	\text{ otherwise}.
\end{cases}
$$

By~\cite[Theorem A]{ag}, $t_G(M)<r_G(M)$ for any irreducible $G$-module $M$, and therefore
\begin{equation}\label{piuuno}h_G(M)\leq \delta_G(M)+1.
\end{equation}

The importance of $h_G(M)$ is clarified by the following proposition.

\begin{prop}[{{\cite[Proposition 2.1]{dl}}}]\label{acca}
	If there exists  an abelian generating chief factor $A$ of $G$, 
	then $ d(G)=h_G(A)$.
\end{prop}

When $G$ admits  a non-abelian generating chief factor $A$, a relation between $\delta_G(A)$ and $d(G)$ is provided by the following result.
\begin{prop}\label{nonabe}
	If $d(G)\geq 3$ and there exists  a non-abelian generating chief factor $A$ of $G$, 
	then $$\delta_G(A)>\frac{|A|^{d(G)-1}}{2|\mathbf{C}_{\aut A}(L_G(A)/A)|}\geq \frac{|A|^{d(G)-2}}{2\log_2|A|}.
	$$
\end{prop}

\begin{proof}
Suppose $d(G)\ge 3$ and let $A$ be a non-abelian generating chief factor of $G$.

For a finite group $X$, let $\phi_X(m)$ denote the number of ordered $m$-tuples $(x_1,\dots,x_m)$ of elements of $X$ generating $X$. Define
\begin{align*}
L&:=L_G(A),\\ 
\gamma&:=|\mathbf{C}_{\aut A}(L/A)|,\\
\delta&:=\delta_G(A),\\
d&:=d(G).
\end{align*}
 In~\cite{dv-l}, it is proved that, if $m\geq d(L),$ then
\begin{equation}\label{crit}d(L_k)\leq m {\text { if and only if }} k\leq \frac{\phi_{L/A}(m)}{\phi_L(m)\gamma}.
\end{equation} By the main result in \cite{uni}, $d(L)=\max(2,d(L/A))$. Since $A$ is a generating chief factor, from the definition,  we have $d(L/A) < d(L_{\delta_G(A)})=d(G)$. As $2 < d(G),$ it follows $d(L)<d(G).$ Now, by applying~\eqref{crit} with $k=\delta_G(A)$ and $m=d(G)-1,$ we deduce
\begin{equation}\label{eq:0}\delta_G(A)>\frac{\phi_{L/A}(d(G)-1)}{\phi_L(d(G)-1)\gamma}.\end{equation}

By~\cite[Corollary 1.2]{dl}, 
\begin{equation}\label{eq:1}\frac{\phi_{L/A}(d(G)-1)}{\phi_L(d(G)-1)}\geq \frac{|A|^{d(G)-1}}{2}.\end{equation}

Moreover, $A\cong S^n$, where $n$ is a positive integer and $S$ is a non-abelian simple group. In the proof
of Lemma 1 in \cite{dv-l-2}, it is shown that $$\gamma \leq n|S|^{n-1}|\aut(S)|.$$
Now,~\cite{ku} shows that $|\out(S)| \leq \log_2(|S|)$ and hence
\begin{equation}\label{eq:2}\gamma\leq n|S|^{n}\log_2(|S|)\leq |S|^{n}\log_2(|S|^n)=|A|\log_2(|A|).\end{equation} From~\eqref{eq:0},~\eqref{eq:1} and~\eqref{eq:2}, we obtain
$$\delta_G(A) > \frac{\phi_{L/A}(d(G)-1)}{\phi_L(d(G)-1)\gamma}\geq \frac{|A|^{d(G)-1}}{2|A|\log_2|A|}=\frac{|A|^{d(G)-2}}{2\log_2|A|}.\qedhere$$
\end{proof}

Recall that $m(G)$ is the largest cardinality of an independent generating set of $G$.
\begin{thm}[{{\cite[Theorem 1.3]{ap}}}]\label{bound} Let $G$ be a finite group. Then $m(G)\geq a+b,$ where $a$ and $b$ are, respectively, the number of non-Frattini and non-abelian factors in a chief series of $G$. Moreover, if $G$ is soluble, then $m(G)=a.$
\end{thm}

\begin{cor}\label{coro}
	Assume that $A$ is the unique minimal normal subgroup of a finite group $G$. If $A$ is non-abelian, then $m(G)\geq 3.$
\end{cor}
\begin{proof}
Suppose first that $G$ is simple. Let $l$ be an element of $G$ of order 2. Since $G=\langle l^x\mid x\in G\rangle$, the set $\{l
^x\mid x \in G\}$ contains a minimal generating set of $G.$ Since $G$ cannot be generated by 2 involutions, this minimal generating set has cardinality at least 3. Thus $m(G)\ge 3$.

Suppose next that $G$ is not simple. Let $a$ and $b$ be the number of non-Frattini and non-abelian factors in a chief series of $G$. As $G$ is not simple,  there exists a maximal normal subgroup $N$ of $G$ containing $A$ and we have a chief series $1\unlhd N_1\unlhd \dots \unlhd N_{t-1}\unlhd N_t=G$ with $N_1=A$ and $N_{t-1}=N.$  Then $a\geq 2$, $b\geq 1$ and $m(G)\geq a+b\geq 3$ by Theorem~\ref{bound}.
\end{proof}

\section{Proof of the main results}
Let $G$ be a finite group, let $d:=d(G)$ and let $m:=m(G)$. Suppose $m=d+1.$ Let $A$ be a generating chief factor of $G$  and let $\delta:=\delta_G(A)$, $L:=L_G(A).$

\subsection{$A$ is non-abelian} First suppose $\delta\geq 2.$ By Theorem \ref{bound}, $m\geq 2\delta$ and therefore $d\geq 2\delta-1\geq 3.$ By Proposition \ref{nonabe}, $$\delta > \frac{|A|^{d-2}}{2\log_2|A|}\geq \frac{
	|A|^{2\delta-3}}{2\log_2|A|}\geq  \frac{60^{2\delta-3}}{2\log_260},$$
but this is never true.

Suppose now $\delta=1$. In this case, by the main theorem in~\cite{uni}, $d=d(L)=\max(2,d(L/A))=2$ and therefore $m=3.$ Since $L$ is an epimorphic image of $G$,  we must have $m(L)\leq 3$. On the other hand $m(L)\geq 3$ by Corollary \ref{coro}. Hence $m(L)=m=3$ and therefore it follows from \cite[Lemma~11]{min} that $G/\frat(G)\cong L.$ Finally, by Theorem \ref{bound}, $m(L)=3$ implies $m(L/A)\leq 1,$ and this is possible only if $L/A$ is a cyclic $p$-group. This concludes the proof of Theorem~\ref{uno}.

\subsection{ $A$ is abelian}
It follows from Proposition~\ref{acca} and~\eqref{piuuno} that $$\delta-1 \leq m-1 = d = h_G(A)\leq \delta+1.$$

If $d=\delta-1,$ then $m=\delta$ and this is possible if $G/\frat(G)\cong A^\delta.$ However, in this case $A$ would be a trivial $G$-module and therefore $d=h_G(A)=\delta=m,$ a contradiction.

Now suppose $d=\delta.$ By Theorem \ref{bound}, $G$ is soluble and contains only one non-Frattini chief factor which is not $G$-isomorphic to $A.$
If $A$ is non-central in $G$, then $G/\frat(G)\cong L_\delta$ and $L/A$ is a cyclic $p$-group, however this  implies $r_G(A)=1,$ $t_G(A)=0$ and $d=h_G(A)=\delta+1,$ which is a contradiction. If $A$ is central, then $G/\frat(G)\cong V\rtimes P$
where $P$ is a finite $p$-group, $V$ is an irreducible  $P$-module and $d(P)=d.$ In particular, we obtain~\eqref{sol1} in Theorem~\ref{thrm:sol}.

 Finally assume $d=\delta+1.$ Notice that in this case $L=A\rtimes H,$ where $A$ is a faithful, non-trivial, irreducible $H$-module, and $$m(H)\leq m-\delta=\delta+2-\delta=2.$$ In particular, by Corollary \ref{coro},  $H$ is soluble. 
 
 If $m(H)=2,$ then $G/\frat(G)\cong L_\delta.$ In particular, we obtain~\eqref{sol2} in Theorem~\ref{thrm:sol}.
 
 If $m(H)=1,$ then there exist two normal subgroups $N_1$  and $N_2$ of $G$ such that $1 \lneq N_1 \leq N_2,$ $G/N_2\cong L_\delta,$ $N_2/N_1\leq \frat(G/N_1)$ and $N_1/\frat(G)$ is an abelian minimal normal subgroup of $G/\frat(G).$ As $m(H)=1$, $H$ is cyclic of prime power order. In particular, we obtain~\eqref{sol3} in Theorem~\ref{thrm:sol}.

\section{Examples for Theorems~\ref{uno} and~\ref{thrm:sol}}\label{sec:mon}
\subsection{Monolithic groups: examples for Theorem~\ref{uno}}

Let $G$  be monolithic primitive with non-abelian socle  $N= S_1\times \dots \times S_n$,
with $S\cong S_i$ for each $1\leq i \leq n.$ The number $\mu(G)=m(G)-m(G/N)$ has been investigated in \cite{mingen}. 
The group $G$ acts by conjugation on the set $\{S_1,\dots,S_n\}$ of the simple components  of $N.$
This produces a group homomorphism
$G\to \sym(n)$ and the image $K$ of $G$ under this homomorphism is a transitive subgroup of $\sym(n).$
Moreover the subgroup $X$  of $\aut S$ induced by the conjugation action of $\textbf{N}_G(S_1)$ on the first
factor $S_1$ is an almost simple group with socle $S.$ 

By \cite [Proposition 4]{mingen} $\mu(G)\geq \mu(X)=m(X)-m(X/S).$ Assume $m(G)=3.$ Observe that by Theorems~\ref{apisa} and~\ref{uno},  
 $G/N$ is cyclic of prime power order. If $X=S$, then 
 \begin{align*}
 3&=m(G)=m(G/N)+\mu(G)\geq m(G/N)+\mu(X)= m(G/N)+m(S)\\
 &\geq m(G/N)+3.
 \end{align*} This implies that $G/N=1,$ and $G=S$ is a simple group. If $X\neq S$, then $G\neq N$ and $$3=m(G)\geq m(G/N)+\mu(G)\geq 1+\mu(X).$$ Moreover $X/S$ is a non-trivial cyclic group of prime power order, so $$m(X)=m(X/S)+\mu(X)\leq 1+\mu(X)\leq 1+2=3.$$ By Corollary~\ref{coro}, $m(X)=3.$

 The groups
 $$\mathrm{P}\Sigma\mathrm{L}_2(9), M_{10},\mathrm{Aut}(\mathrm{PSL}_2(7))$$ are the only known examples (to the authors of this paper) of almost simple groups $X$ with $X\neq \soc(X)$ and $m(X)=3.$ We do believe that there are other such examples, but our current computer codes are not good enough to have a thourough investigation.

 Let $S:=\mathrm{PSL}_2(7)$   and consider
 the wreath product $W:=\aut(S) \wr \perm(n)$. Any element $w\in W$ can be written as $w=\pi(a_1,\dots,a_n),$ with $\pi\in \perm(n)$ and $a_i\in \aut(S)$ for $1\leq i\leq n$. In particular $N=\soc(W) \cong S_1\times \dots\times S_n=\{(s_1,\dots,s_n)\mid s_i \in S\}$.

\begin{prop}
Let $G$ be the subgroup of $W$ generated by $N=\soc(W)$ and $\gamma=\sigma(a,1,\dots,1),$ where $\sigma=(1,2,\dots,n)\in \perm(n)$ and $a\in \aut(S)\setminus S.$ If $n=2^t$ for some positive integer $t,$ then
$m(G)=3.$
\end{prop}
In particular, this gives infinitely many examples of non-simple, non-soluble groups $G$ with $m(G)=d(G)+1$ in Theorem~\ref{uno}.  Similar examples can be obtained replacing $\mathrm{Aut}(\mathrm{PSL}_2(7))$ with $\mathrm{P}\Sigma\mathrm{L}_2(9)$ 
or $M_{10}$.
\begin{proof}
Suppose $n=2^t$, for some positive integer $t$. Let $r:=m(G)$; we aim to prove that $r=3$.

 Let
$\{g_1,\dots,g_r\}$ be an independent generating set of $G.$ Observe that 
$$\gamma^n=(a,\ldots,a)\in G\setminus N$$ and hence  $G/N$ is cyclic of order $2^{t+1}.$
Therefore, relabelling the elements of the independent generating set if necessary, we may assume $G=\langle g_1, N\rangle$.
Hence $g_1=\sigma(as_1,s_2,\dots,s_n)$ with $s_1,\dots,s_n\in S.$ Moreover, for $2\leq i\leq r,$ there exists $u_i\in \mathbb Z$ such that $g_ig_1^{u_i}\in N$. Observe that $\{g_1,g_2g_1^{u_2},\dots,g_rg_1^{u_r}\}$ is still an independent generating set having cardinality $r$.

 Let $$m=(s_2\cdots s_n, s_3\cdots s_n,\dots,s_{n-1}s_n,s_n,1) \in N.$$
Then $Y=\{g_1^m,(g_2g_1^{u_2})^m,\dots,(g_rg_1^{u_r})^m\}$ is another independent generating set for $G$ having cardinality $r$. 
We have
$$y_1:=g_1^m=\sigma(b,1,\dots,1),$$ with $b=as_1\cdots s_n \in \aut S\setminus S,$ and, for $2\leq i\leq r,$ there exist $s_{i1},\dots s_{in}\in S$ such that
$$y_i:=(g_ig_1^{u_i})^m=(s_{i1},\dots,s_{in}).$$

Now let $Z:=\{b, s_{ij} \mid 2\leq i\leq r, 1\leq j\leq n\}$ and $T=\langle Z\rangle.$ Since $G=\langle y_1,\dots,y_t\rangle \leq T \wr \langle \sigma \rangle,$ we must have $\aut(S)=T.$ On the other hand
$m(\aut(S))=3,$ so $\aut(S)=\langle b, s_{iu}, s_{jv}\rangle,$ for suitable $2\leq i,j\leq r$ and $2\leq u, v\leq n.$ 

Let $H:=\langle y_1,  y_i, y_j\rangle$ and, for $1\leq k\leq n,$ consider the projection $\pi_k: N\to S$ sending $(s_1,\dots,s_n)$ to $s_k.$ Notice that $\pi_1(y_1^n)=b,$ $\pi_1((y_i)^{y_1^{1-u}})=s_{iu},$ $\pi_1((y_j)^{y_1^{1-v}})=s_{jv}.$ 
In particular $\pi_1(H\cap N)=S$ and $H\cap N$ is a subdirect product of 
$N= S_1\times \dots \times S_n$.

Recall that a subgroup $D$ of $N=S_1\times \dots \times S_n$ is call full diagonal  if each projection $\pi_i: D\to S_i$ is an isomorphism.
To each pair $(\Phi,\alpha)$ where $\Phi=\{B_1,\dots,B_c\}$ is a partition of the set $\{1,\dots,n\}$
and  $\alpha=(\alpha_1,\dots,\alpha_n)\in (\aut S)^n$, we associate a direct product
$\Delta(\Phi,\alpha)=D_1\times \dots \times D_c$ where each $D_j=\{(x^{\alpha_{i_1}},\dots, x^{\alpha_{i_d}})\mid x\in S\}$ is a full diagonal subgroup of the direct product $S_{i_1} \times \dots \times S_{i_d}$ corresponding to the block $B_j=\{i_1,\dots,i_d\}$ in $\Phi.$ 

Since $H\cap N$ is a subdirect product of $N,$ we must have $H\cap N=\Delta(\Phi,\alpha)$ for a suitable choice of the pair $(\Phi,\alpha).$ As $G=\langle H,N\rangle$, the action by conjugation of $H$ on $\{S_1,\ldots,S_n\}$ is transitive and hence the partition $\{B_1,\dots,B_c\}$ correspond to an imprimitive system for the permutation action of $\langle\sigma\rangle$ on $\{1,\dots,n\}.$
So there exist $c=2^\gamma$ and $d=2^\delta$ with $c\cdot d=n$ 
such that $$B_i:=\{i,i+c,i+2c,\dots,i+(d-1)c\} \text { for $1\leq i \leq c$.}$$
Notice that $y_1\in H$ normalizes $\Delta(\Phi,\alpha)$. In particular, $y_1^c$ normalizes $\Delta(\Phi,\alpha).$ However $y_1^c$ normalizes $L=S_1\times S_{1+c}\times \dots \times S_{1+(d-1)c}$ and acts on $L$ as $\pi \cdot l,$ 
where $\pi$ is the $d$-cycle $(1,1+c,\dots,1+(d-1)c)$ and $l=(b,1,\dots,1)\in L.$ 
In particular $\pi\cdot l$ normalizes the full diagonal subgroup $D_1$ of $L$. Therefore, setting $\phi_i=\alpha_{1+(i-1)c}$, for every $s\in S$, there exists $t\in T$ such that
	$$(s^{\phi_db},s^{\phi_1},s^{\phi_2},\ldots,s^{\phi_{d-1}})=
	(t^{\phi_1},t^{\phi_2},t^{\phi_3
	},\ldots,t^{\phi_{d}}).$$ It follows
	$$\begin{aligned}
	\phi_{d}b\phi_1^{-1}\phi_2&=\phi_1,\\
		\phi_{d}b\phi_1^{-1}\phi_3&=\phi_{2},\\
		\cdots\\
		\phi_{d}b\phi_1^{-1}\phi_d&=\phi_{d-1}.
	\end{aligned}$$
In particular $(\phi_1\phi_d^{-1})^d\equiv b^{d-1}\! \mod S.$ If $d$ is even, then 
$b\in \langle x^2\mid x\in \aut(S)\rangle=S,$ against our assumption. Thus $d=1$ and hence $c=n$. However, this implies $H\cap N=N$ and consequently $H=G.$ Thus $m(G)=r\leq 3$. On the other hand $m(G)\geq 3$ by Theorem \ref{bound}. So we conclude $m(G)=3.$
\end{proof}
\subsection{Soluble groups: examples for Theorem~\ref{thrm:sol}}
We give three elementary examples, but with the same ideas one can construct more complicated examples.

Let $S_n$ be the symmetric group of degree $n$ and let $C_n$ be the cyclic group of order $n$.

Now the group $G:=S_3 \times C_2^t = C_3 : C_2^{t+1}$ with $t\ge 1$ satisfies $d(G)=t+1$ and $m(G)=t+2$. This gives examples of groups satisfying~\eqref{sol1} in Theorem~\ref{thrm:sol}.

The group $G:=S_4=K:S_3$ with $K$ the Klein subgroup of $S_4$ and the group $G:=(C_3^t : C_2) \times C_2$ with $C_2$ acting on  $C_3^t$ by inversion also satisfy $m(G)=d(G)+1$. These two examples show groups satisfying~\eqref{sol2} in Theorem~\ref{thrm:sol} with $m(H)=2$ in the first case and with $H$ abelian in the second case.

As above, let $K$ be the Klein subgroup of $S_4$ and let  $G:=K:(S_3\times C_2^{t-1})$. This gives examples of groups satisfying~\eqref{sol3} in Theorem~\ref{thrm:sol}.

\end{document}